\documentclass[11pt,a4paper]{article}

\usepackage{amsfonts,amsgen,amsmath,amstext,amsbsy,amsopn,amsthm,mathrsfs,comment}
\usepackage{tikz,multirow}

\usetikzlibrary{decorations.markings}
\tikzstyle{vertex}=[circle, draw, inner sep=0pt, minimum size=4pt]
\newcommand{\vertex}{\node[vertex]}

\newtheorem{thm}{Theorem}
\newtheorem{claim}{Claim}
\newtheorem{lemma}{Lemma}
\newtheorem{prob}{Problem}
\newtheorem{remark}{Remark}
\newtheorem{cor}{Corollary}
\newtheorem{prop}{Proposition}

\begin{document}

\title{\bf\Large On sufficient conditions for rainbow cycles
in edge-colored graphs}

\date{}

\author{Shinya Fujita\footnote{School of Data Science, Yokohama
City University, 22-2, Seto, Kanazawa-ku,
Yokohama, 236-0027, Japan. E-mail:~fujita@yokohama-cu.ac.jp (S. Fujita).}~~
~Bo Ning\thanks{Corresponding author. Center for Applied Mathematics,
Tianjin University, Tianjin, 300072, P.R. China. E-mail:~bo.ning@tju.edu.cn
(B. Ning).}~~~Chuandong Xu\thanks{School of Mathematics and Statistics,
Xidian University, Xi'an, 710071, P.R. China. E-mail: xuchuandong@xidian.edu.cn
(C. Xu).}~~~Shenggui Zhang\thanks{Department of Applied Mathematics,
Northwestern Polytechnical University,~Xi'an, Shaanxi, 710072,
P.R.~China. E-mail:~sgzhang@nwpu.edu.cn~(S. Zhang).}}
\maketitle
\begin{abstract}
Let $G$ be an edge-colored graph. We use $e(G)$ and $c(G)$ to denote the
number of edges of $G$ and the number of colors appearing on $E(G)$,
respectively. For a vertex $v\in V(G)$,
the \emph{color neighborhood} of $v$ is defined as the
set of colors assigned to the edges incident to $v$.
A subgraph of $G$ is \emph{rainbow} if all of its
edges are assigned with distinct colors. The well-known Mantel's theorem states
that a graph $G$ on $n$ vertices contains a triangle if
$e(G)\geq\lfloor\frac{n^2}{4}\rfloor+1$.
Rademacher (1941) showed that $G$ contains at least $\lfloor\frac{n}{2}\rfloor$
triangles under the same condition. Li, Ning, Xu and Zhang (2014)
proved a rainbow version of Mantel's theorem: An edge-colored
graph $G$ has a rainbow triangle if $e(G)+c(G)\geq n(n+1)/2$.
In this paper, we first characterize all graphs $G$ satisfying
$e(G)+c(G)\geq n(n+1)/2-1$ but containing no rainbow triangles.
Motivated by Rademacher's theorem,  we then characterize
all graphs $G$ which satisfy $e(G)+c(G)\geq n(n+1)/2$
but contain only one rainbow triangle. We further obtain two results on
color neighborhood conditions for the existence of rainbow
short cycles. Our results improve a previous theorem due to
Broersma, Li, Woeginger, and Zhang (2005). Moreover, we provide a sufficient condition
in terms of color neighborhood
for the existence of a specified number of vertex-disjoint
rainbow cycles.

\smallskip
\end{abstract}

\section{Introduction}
Let $G$ be a graph. We use $V(G)$
and $E(G)$ to denote the vertex set and edge set of $G$,
respectively, and call $|G|:=|V(G)|$ and $e(G):=|E(G)|$ the \emph{order}
and the \emph{size} of $G$. For a subset $S$ of $V(G)$,
we use $G[S]$ to denote the subgraph of $G$ induced by $S$,
and $G-S$ to denote the subgraph $G[V(G)\backslash S]$.
When $V(H)=\{v\}$, we use $G-v$ instead of $G-\{v\}$.
For disjoint subsets $S,S'$ of $V(G)$, let $G[S,S']$
denote the bipartite subgraph of $G$ induced by $S$
and $S'$, i.e., $G[S,S']$ has classes $S,S'$
and edge set $\{xy\in E(G):x\in S, y\in S'\}$.

An edge-coloring of $G$ is a mapping $C: E(G)\rightarrow\mathbb{N}$,
where $\mathbb{N}$ is the set of all natural numbers. When $G$ has
such a coloring, we call it an \emph{edge-colored graph}. Let $G$ be an
edge-colored graph. We use $C(G)$ to denote the set of colors
appearing on the edges of $G$ and let $c(G):=|C(G)|$. For a vertex
$v\in V(G)$ and a subgraph $H$ of $G$, the \emph{color neighborhood}
of $v$ in $H$, denoted by $CN_H(v)$,
is defined as the set of colors assigned to the edges from
$v$ to $V(H)\backslash \{v\}$. The \emph{color degree} of $v$ in $H$
is denoted by $d_{H}^c(v):=|CN_H(v)|$; and the \emph{minimum color degree} of $G$,
denoted by $\delta^{c}(G)$, is equal to $\min\{d_{G}^{c}(v):v\in V(G)\}$.
When there is no fear of confusion,
we write $CN(v)$ and $d^c(v)$ instead of $CN_G(v)$ and $d^c_G(v)$ for short,
respectively. An edge-colored graph is \emph{rainbow} if all its edges receive distinct colors,
and \emph{monochromatic} if all its edges have the same color.
We use Bondy and Murty \cite{BM08}, and Chartrand and
Zhang \cite{CZ08} for notation and terminology not defined here.
For more results on related topics on rainbow subgraphs, we refer the reader to surveys due to Kano and
Li \cite{KL09}, and Fujita, Magnant and Ozeki \cite{FMO10,FMO14}.

We first recall some classical result on the existence of short
cycles in uncolored graphs. Mantel's theorem (1907) is one important
starting point of extremal graph theory, which is stated as every graph
$G$ on $n$ vertices contains a triangle if
$e(G)\geq \lfloor\frac{n^2}{4}\rfloor$, unless
$G\cong K_{\lceil n/2\rceil,\lfloor n/2\rfloor}$.
Li et al. \cite{LNXZ14} obtained a rainbow version of
Mantel's theorem.

\begin{thm}[Li, Ning, Xu, and Zhang \cite{LNXZ14}]\label{thm_rainbowK3}
Let $G$ be an edge-colored graph of order $n\geq 3$.
If $e(G)+c(G)\geq n(n+1)/2$, then $G$ contains a rainbow $C_3$.
\end{thm}

The bound for $e(G)+c(G)$ in the above theorem is best
possible. To see this, let $\mathcal{G}_0$ be the set of all edge-colored
complete graphs which satisfy the following properties
(see Figure 1):
\begin{enumerate}
\item $K_1\in \mathcal{G}_0$;
\item For every $G\in \mathcal{G}_0$ of order
$n\geq2$, $c(G)=n-1$ and there is a bipartition
$V(G)=V_1\cup V_2$, such that $G[V_1,V_2]$ is
monochromatic and $G[V_i]\in \mathcal{G}_0$ for $i=1,2$.
\end{enumerate}
One can check that every graph in
$\mathcal{G}_0$ satisfies that
$e(G)+c(G)\geq n(n+1)/2-1$ but
contains no rainbow triangles.

In this paper we firstly characterize all the
graphs which satisfy $e(G)+c(G)\geq n(n+1)/2-1$
but contain no rainbow triangles. Our result
shows that all extremal graphs are included in $\mathcal{G}_0$.

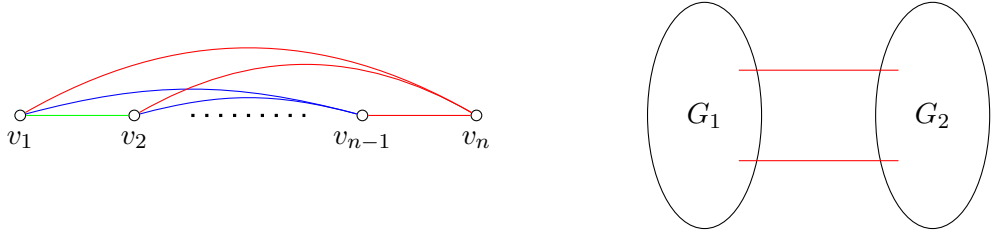
\begin{figure}[bht]
\centering
\begin{tikzpicture}[x=1.5cm, y=1.5cm]
\vertex (v1) at (-5,0) [label=below:$v_1$]{};
\vertex (v2) at (-4,0) [label=below:$v_2$]{};
\vertex (v3) at (-2,0) [label=below:$v_{n-1}$]{};
\vertex (v4) at (-1,0) [label=below:$v_{n}$]{};
	
\path (v1) edge[color=green] (v2);
\path (v1) edge[color=blue,bend left=15] (v3);
\path (v1) edge[color=red,bend left=30] (v4);
\path (v2) edge[color=blue,bend left=15] (v3);
\path (v2) edge[color=red,bend left=30] (v4);
\path (v3) edge[color=red] (v4);
	
\draw[loosely dotted, very thick] (-3.5,0) -- (-2.5,0);
	
\draw (1,0) ellipse (0.5 and 1) node  {$G_1$};
\draw (3,0) ellipse (0.5 and 1) node  {$G_2$};
	
\path (1.3,0.4) edge[color=red] (2.7,0.4);
\path (1.3,-0.4) edge[color=red] (2.7,-0.4);
\end{tikzpicture}
\caption{An example in $\mathcal{G}_0$ and the
 structure of graphs in $\mathcal{G}_0$ for $n\geq 2$.}
\end{figure}

\begin{thm}\label{theorem_G0}
Let $G$ be an edge-colored graph of order $n$.
If $e(G)+c(G)\geq \binom{n+1}{2}-1$ and $G$ contains
no rainbow triangles, then $G$ belongs to $\mathcal{G}_0$.
\end{thm}

In 1941, an extension of Mantel's theorem
was obtained by Rademacher in an unpublished manuscript (see \cite{E55}).
He proved that every graph $G$ on $n$ vertices contains at least
$\lfloor n/2\rfloor$ triangles if
$e(G)\geq \lfloor\frac{n^2}{4}\rfloor+1$. So, one may naturally
ask whether there is a rainbow version of Rademacher's theorem.
The following example shows that the answer is no.

Let $\mathcal{G}_1$ be the set of all edge-colored
complete graphs which satisfy the following properties:
\begin{enumerate}
\item The rainbow $C_3$ is included in $\mathcal{G}_1$;
\item For every $G\in \mathcal{G}_1$ of order
$n\geq 4$, $c(G)=n$ and there is a bipartition
$V(G)=V_1\cup V_2$, such that $G[V_1,V_2]$ is monochromatic
and $G_1=G[V_1]\in \mathcal{G}_1$, $G_2=G[V_2]\in \mathcal{G}_0$.
\end{enumerate}

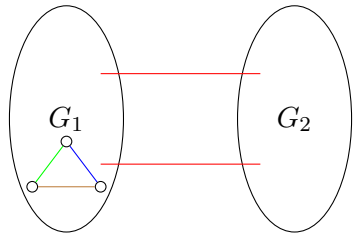
\begin{figure}[bht]
\centering
\begin{tikzpicture}[x=1.5cm, y=1.5cm]
\vertex (v1) at (1,-0.2) []{};
\vertex (v2) at (0.7,-0.6) []{};
\vertex (v3) at (1.3,-0.6) []{};
\path (v1) edge[color=green] (v2);
\path (v1) edge[color=blue] (v3);
\path (v2) edge[color=brown] (v3);
%
%
\draw (1,0) ellipse (0.5 and 1) node  {$G_1$};
\draw (3,0) ellipse (0.5 and 1) node  {$G_2$};
\path (1.3,0.4) edge[color=red] (2.7,0.4);
\path (1.3,-0.4) edge[color=red] (2.7,-0.4);
\end{tikzpicture}
\caption{The structure of graphs in $\mathcal{G}_1$ for $n\geq 4$.}
\end{figure}

Generally, let $\mathcal{G}_k$ $(k\geq 2)$ be the set of all edge-colored complete graphs
constructed as follows: For every $G\in \mathcal{G}_k$ of order $n\geq 3k$,
there is a bipartition $V(G)=V_1\cup V_2$, such that $G[V_1,V_2]$ is monochromatic and $G[V_1]\in \mathcal{G}_i$,
$G[V_2]\in \mathcal{G}_{k-i}$ for some $0\leq i\leq k$. It is easy to see that for every $G\in \mathcal{G}_k$,
$e(G)=c(G)=n(n-1)+(k-1)$ and $G$ contains exactly $k$ rainbow triangles. For $k=1$, we can show $\mathcal{G}_1$ is
exactly the set of graphs which satisfy such properties.

\begin{thm}\label{theorem_G1}
Let $G$ be an edge-colored graph of order $n\geq 3$.
If $e(G)+c(G)\geq\binom{n+1}{2}$ and $G$ contains
exactly one rainbow triangle, then $G$ belongs to
$\mathcal{G}_1$.
\end{thm}

Aside from the color number condition in Theorem 1, Li et al. \cite{LNXZ14}
also considered a Dirac-type color degree condition for the existence of rainbow
triangles in edge-colored graphs.

\begin{thm}[Li, Ning, Xu, and Zhang \cite{LNXZ14}]\label{theorem_RainbowC_3}
Let $G$ be an edge-colored graph of order $n\geq 5$. If $d^{c}(v)\geq n/2$ for every
vertex $v\in V(G)$ and $G$ contains no rainbow $C_3$, then the underlying graph
of $G$ is $K_{n/2,n/2}$, where $n$ is even.
\end{thm}

Returning to related topics in uncolored graphs, let us recall the
Ore-type condition, that is, the condition in terms of the minimum degree
sum of non-adjacent vertices in a graph (see e.g. \cite{West}).
This kind of condition was introduced as an extension of the minimum
degree condition for cycles, thereby yielding affluent results
in this area. Motivated by this, when we try to
consider some natural extensions from the minimum
color degree condition in edge-colored graphs, what
kind of color degree condition would be appropriate?

Perhaps the following theorem due to Broersma et al. \cite{BLWZ05}
gives us a reasonable answer to this question.

\begin{thm}[Broersma, Li, Woeginger, and Zhang \cite{BLWZ05}]\label{BLWZ}
Let $G$ be an edge-colored graph of order
$n\geq 4$ such that $|CN(u)\cup CN(v)|\geq n-1$
for every pair of vertices $u$ and $v$ in
$V(G)$. Then $G$ contains a rainbow $C_3$ or
a rainbow $C_4$.
\end{thm}

Unlike Ore-type conditions in uncolored graphs, we
look at every pair of vertices in the edge-colored graph $G$
under the assumption of Theorem \ref{BLWZ}.
This is because we need to deal with the case that $G$ is
an edge-colored complete graph, and even in this special
case, problems for finding rainbow cycles are far from trivial
in general (unlike the uncolored version). An example is a theorem
by Li et al. \cite{LNZ16} which states that an edge-colored graph
on $n$ vertices contains a rainbow triangle if the color
degree sum of every two adjacent vertices is at least $n+1$.

Motivated by Theorem~\ref{BLWZ}, one may naturally ask whether we can find both
a rainbow $C_3$ and a rainbow $C_4$ under the same condition.
The following theorems answer the above
question affirmatively in some sense.

\begin{thm}\label{theorem_colorneighbor_C_4}
Let $k$ be a positive integer, and $G$ an edge-colored
graph of order $n\geq 105k-24$ such that $|CN(u)\cup CN(v)|\geq n-1$
for every pair of vertices $u$ and $v$ in $V(G)$.
Then $G$ contains $k$ rainbow $C_4$'s.
\end{thm}

\begin{thm}\label{theorem_colorneighbor_C_3}
Let $G$ be an edge-colored graph of order
$n\geq 6$ such that $|CN(u)\cup CN(v)|\geq n-1$
for every pair of vertices $u$ and $v$ in
$V(G)$. Then $G$ contains a rainbow $C_3$ unless $G$
is a rainbow $K_{\lceil n/2\rceil,\lfloor n/2\rfloor}$.
\end{thm}

So far, we have introduced some results on the existence of rainbow
short cycles in edge-colored graphs. As observed, we need quite a
strong assumption to guarantee the existence of rainbow short cycles.
Similarly, when we consider a degree condition for the existence
of small cycles in uncolored graphs, it becomes a strong assumption. However, this is not the case
if we just want to find a cycle with no restriction on its length in uncolored graphs.
In contrast to this uncolored case, the situation might not change drastically even if we just
hope for the existence of rainbow cycles with no restriction on
their lengths in edge-colored graphs. Yet we could improve the coefficient of $n$ in the
assumption of Theorem~\ref{BLWZ} from $1$ to $1/2$, if we do not
restrict the length of rainbow cycles. Moreover, we could strengthen the
conclusion part as ``vertex-disjoint" rainbow cycles.

\begin{thm}\label{theorem_colorneighbor_verdiscyc}
Let $k$ be a positive integer. If an edge-colored
graph $G$ of order $n$ satisfies $|CN(u)\cup CN(v)|\geq n/2+64k+1$ for every pair of vertices
$u,v\in V(G)$, then $G$ contains $k$ vertex-disjoint rainbow cycles.
\end{thm}

By this theorem, we obtain the following corollary, although
Theorem~\ref{theorem_RainbowC_3} already implies it as well.

\begin{cor}\label{C}
Let $k$ be a positive integer. If an edge-colored graph
$G$ of order $n$ satisfies $\delta^c(G)\ge n/2+64k+1$, then $G$ contains $k$
vertex-disjoint rainbow cycles.
\end{cor}

Comparing with the color degree conditions under the assumptions of
Theorem~\ref{theorem_colorneighbor_verdiscyc} and Corollary~\ref{C},
we can observe that our theorem provides a substantial extension
in view of color degree conditions for the existence of vertex-disjoint
rainbow cycles.

The organization of this paper is as follows. In Section 2,
we prove Theorems \ref{theorem_G0} and \ref{theorem_G1}.
In Section 3, we prove Theorems \ref{theorem_colorneighbor_C_4},
\ref{theorem_colorneighbor_C_3} and \ref{theorem_colorneighbor_verdiscyc}.
We conclude this paper with some remarks and problems.

\section{Proofs of Theorems \ref{theorem_G0} and \ref{theorem_G1}}
Before giving the proofs, we first introduce a concept given in \cite{LNXZ14}.
Let $v$ be a vertex in an edge-colored graph $G$. A color $c$ is
\textit{saturated} by $v$ if all the edges with the color $c$ are
incident to $v$. In this case,  $c\notin C(G-v)$. As in
\cite{LNXZ14}, the \emph{color saturated degree} of $v$ is defined as
$d^s(v):=c(G)-c(G-v)$.

\begin{lemma}[Li, Ning, Xu, and Zhang~\cite{LNXZ14}]\label{lemma_sdegree}
Let $G$ be an edge-colored graph. Then
$\sum\limits_{v\in V(G)} d^s(v)\leq 2c(G)$,
and the equality holds {if and only if} $G$ is rainbow.
\end{lemma}

\begin{lemma}\label{lemma_complete0}
Let $G$ be an edge-colored graph of order $n\geq 2$. If
$e(G)+c(G)=\binom{n+1}{2}-1$ and $G$ contains no rainbow
triangle, then $G$ is complete and contains a vertex $u$
such that $d^s(u)=1$.
\end{lemma}

\begin{proof}
We prove this lemma by induction on the order of $G$.
It is trivial that the result holds for $n=2,3$. Now
assume that it holds for a graph with order smaller than $n$,
where $n\geq 4$.
	
\begin{claim}
For every $v\in V(G)$, $d(v)+d^s(v)\geq n$.
\end{claim}
	
\begin{proof}
Suppose not. Then there exists a vertex $v\in V(G)$ satisfying
$d(v)+d^s(v)\leq n-1$. This implies that $e(G-v)+c(G-v)=e(G)+c(G)-d(v)-d^s(v)\geq \binom{n}{2}$.
It follows from Theorem \ref{thm_rainbowK3} that $G-v$ contains a
rainbow triangle, a contradiction.
\end{proof}
	
\begin{claim}
There exists a vertex $u\in V(G)$ such that $d(u)+d^s(u)= n$.
\end{claim}
	
\begin{proof}
Suppose not. Then $d(v)+d^s(v)\geq n+1$ for every $v\in V(G)$. It
follows from Lemma \ref{lemma_sdegree} that
\[
n(n+1)\leq \sum_{v\in V(G)}\bigl(d(v)+d^s(v)\bigr)\leq 2e(G)+2c(G)=n(n+1)-2,
\]
a contradiction.
\end{proof}
	
It is easy to see that $e(G-u)+c(G-u)=e(G)+c(G)-d(u)-d^s(u)=\binom{n}{2}-1$.
By the induction hypothesis, the graph $G-u$ is complete.
	
If $d(u)< n-1$ then $d^s(u)\geq 2$. Let $uv,uw$ be two edges with
distinct colors which are saturated by $u$. By the definition of
saturated colors, neither $C(uv)$ nor $C(uw)$ appears in $G-u$.
Thus, $uvwu$ is a rainbow triangle, a contradiction. It follows
that $d(u)=n-1$ and $d^s(u)=1$. Thus, $G$ is complete and $d^s(u)=1$.
This proves Lemma \ref{lemma_complete0}.
\end{proof}

A \emph{Gallai coloring} is an edge-coloring of the complete graph $K_n$ such that
there are no rainbow triangles in it. (See the references in \cite{GS04}.)
The following two classical theorems on Gallai colorings play an important
role in the proof of Theorem \ref{theorem_G0}.

\begin{lemma}[Gy\'arf\'as and Simonyi~\cite{GS04}]\label{lemma_Gallaicoloring}
Any Gallai coloring can be obtained by substituting complete graphs with
Gallai colorings into vertices of 2-edge-colored complete graphs
with at least two vertices.

\end{lemma}
\begin{lemma}[Erd\H{o}s, Simonovits, and S\'{o}s~\cite{ESS73}]\label{lemma_K_ncoloring}
Any Gallai coloring of $K_n$ can use at most $n-1$ colors.
\end{lemma}

\smallskip
\noindent
{\bf {Proof of Theorem \ref{theorem_G0}.}}
We prove this result by induction on the
order of $G$. Obviously, the result
holds for $n=1,2,3$. Now assume that it
holds for any graph with order smaller than $n\geq 4$.
	
By Theorem \ref{thm_rainbowK3}, we can assume that
$e(G)+c(G)=\binom{n+1}{2}-1$. It follows from Lemma \ref{lemma_complete0} that
$G$ is complete. Since $e(G)+c(G)=\binom{n+1}{2}-1$,
$c(G)=n-1$. Thus the edge-coloring of $G$ is
a Gallai coloring with $n-1$ colors. By Lemma \ref{lemma_Gallaicoloring},
the coloring of $G$ can be obtained by substituting complete
graphs $H_1,H_2,\ldots,H_k$ with Gallai colorings
into vertices of a 2-edge-colored complete graph $K_k$, where $k\geq 2$,
and $|H_i|=n_i$, $i=1,2,\ldots,k$. Note that $\sum_{i=1}^kn_i=n$.
By Lemmas \ref{lemma_Gallaicoloring} and \ref{lemma_K_ncoloring},
\[
c(G)\leq \sum_{i=1}^{k}c(H_i)+2\leq \sum_{i=1}^{k}(n_i-1)+2=n-k+2.
\]
On the other hand,  $c(G)=n-1$. Thus $k=2,3$.

It is easy to see that every 2-edge-colored $K_k$ has a monochromatic
cut for $k=2,3$. By Lemma \ref{lemma_Gallaicoloring}, there
is also a monochromatic cut in $G$. Let $V_1,V_2$ be the classes
of this monochromatic cut.
It follows from Lemma \ref{lemma_K_ncoloring} that
\[
n-1=c(G)\leq c(G[V_1])+c(G[V_2])+c(G[V_1,V_2])\leq (|V_1|-1)+(|V_2|-1)+1=n-1.
\]
This implies that
\[
c(G)= c(G[V_1])+c(G[V_2])+c(G[V_1,V_2]),
\]
which holds if and only if $C(G[V_1])$, $C(G[V_2])$ and $C(G[V_1,V_2])$
are pairwise disjoint sets. Moreover,
\[
c(G[V_1])=|V_1|-1,\quad c(G[V_2])=|V_2|-1\quad \text{and}\quad c(G[V_1,V_2])=1.
\]
By the induction hypothesis, both $G[V_1]$ and $G[V_2]$ belong to $\mathcal{G}_0$.
It follows from the definition of $\mathcal{G}_0$ that $G\in \mathcal{G}_0$.

The proof is complete.{\hfill$\Box$}

\smallskip
The proof of Theorem \ref{theorem_G1} is based on the following two lemmas.

\begin{lemma}[Rademacher \cite{E55}]\label{lemma_CountK3s}
Let $G$ be a graph with order $n$ and size $m$.
If $m\geq \lfloor \frac{n^2}{4}\rfloor+1$, then $G$ contains at least
$\lfloor \frac{n}{2}\rfloor$ triangles.
\end{lemma}

\begin{lemma}\label{lemma_complete1}
Let $G$ be an edge-colored graph of order $n\geq3$.
If $e(G)+c(G)\geq\binom{n+1}{2}$ and $G$ contains exactly
one rainbow triangle, then $e(G)+c(G)=\binom{n+1}{2}$
and $G$ is complete.
\end{lemma}

\begin{proof}
We prove this result by induction on the order of $G$.
It is trivial for $n=3$. Now we assume that the lemma holds
for any graph of order smaller than $n\geq 4$. Denote
by $v_1v_2v_3v_1$ the unique rainbow triangle in $G$.
Let $V_1=\{v_1,v_2,v_3\}$ and $V_2=V(G)\backslash V_1$.
	
\setcounter{claim}{0}
\begin{claim}
$G$ is not rainbow.
\end{claim}
	
\begin{proof}
Suppose that $G$ is rainbow. Then
$e(G)=c(G)\geq \frac{n^2}{4}+\frac{n}{4}\geq \lfloor \frac{n^2}{4}\rfloor+1$.
It follows from Lemma \ref{lemma_CountK3s} that
$G$ contains at least $\lfloor n/2\rfloor\geq 2$ triangles,
which are rainbow triangles in $G$, a contradiction.
\end{proof}
	
\begin{claim}
$e(G)+c(G)=\binom{n+1}{2}$.
\end{claim}
	
\begin{proof}
Suppose that $e(G)+c(G)\geq\binom{n+1}{2}+1$.
Let $e$ be an edge in the unique rainbow triangle
of $G$. Then $G-e$ contains no rainbow triangle, and
\[
e(G-e)+c(G-e)\geq (e(G)-1)+(c(G)-1)\geq \binom{n+1}{2}-1.
\]	
It follows from Theorem \ref{theorem_G0} that $G-e$ is
complete, a contradiction.
\end{proof}
	
\begin{claim}
For every $v\in V_1$, $d(v)+d^s(v)\geq n+1$; for every
$v\in V_2$, $d(v)+d^s(v)\geq n$.
\end{claim}
	
\begin{proof}
For every $v\in V_1$, $G-v$ contains no rainbow triangle.
It follows from Theorem \ref{thm_rainbowK3} that
$e(G-v)+c(G-v)\leq \binom{n}{2}-1$. Thus $d(v)+d^s(v)\geq n+1$.
		
Suppose that there exists a vertex $u\in V_2$ such that
$d(u)+d^s(u)\leq n-1$. Then $G-u$ contains a unique
rainbow triangle and $e(G-u)+c(G-u)\geq \binom{n}{2}+1$.
It follows from the induction hypothesis that
$e(G-u)+c(G-u)=\binom{n}{2}$, a contradiction.
\end{proof}
	
\begin{claim}
There exists a vertex $u\in V_2$ such that $d(u)+d^s(u)=n$.
\end{claim}
	
\begin{proof}
Suppose not. Then, $d(v)+d^s(v)\geq n+1$ for every $v\in V_2$.
By Claim 3 and Lemma \ref{lemma_sdegree},
\[
n(n+1)\leq \sum_{v}\bigl(d(v)+d^s(v)\bigr)\leq 2e(G)+2c(G)=n(n+1).
\]
Thus $\sum_v d^s(v)=2c(G)$. It follows from Lemma
\ref{lemma_sdegree} that $G$ is rainbow, a contradiction to Claim 1.
\end{proof}
	
Let $u$ be as in Claim 4. Note that $G-u$ contains exactly
one rainbow triangle and
\[
e(G-u)+c(G-u)=e(G)+c(G)-d(u)-d^s(u)=\binom{n}{2}.
\]
It follows from the induction hypothesis that $G-u$ is complete.
	
Now we show that $d(u)=n-1$. Suppose that $d(u)<n-1$.
Then, we obtain $d^s(u)\geq 2$. Let $uv$ and $uw$ be
two edges with distinct colors which are saturated by
$u$. It is easy to see that $uvwu$ is a rainbow triangle
distinct from $v_1v_2v_3v_1$, a contradiction.
Thus, $G$ is complete, and together with Claim 2,
this proves Lemma \ref{lemma_complete1}.
\end{proof}

\smallskip
\noindent
{\bf {Proof of Theorem \ref{theorem_G1}.}}
We prove this result by induction on the order
of $G$. It is trivial for $n=3$. Now assume
that the theorem holds for graphs with order smaller
than $n\geq 4$. Denote by $v_1v_2v_3v_1$ the unique
rainbow triangle in $G$.
	
We show that $C(v_1v_2),C(v_1v_3)$
are saturated by the vertex $v_1$. It follows from Claim 3 (in
the proof of Lemma \ref{lemma_complete1}) that $d(v_i)+d^s(v_i)\geq n+1$
for each $i=1,2,3$, and hence $d^s(v_i)\geq 2$ for each $i=1,2,3$.
First, suppose that there is exactly one color in $\{C(v_1v_2),C(v_1v_3)\}$,
say $C(v_1v_2)$, which is saturated by $v_1$. Since $d^s(v_1)\geq 2$, we can
choose $w\in N(v_1)$ such that $w\neq v_2$, $C(v_1w)\neq C(v_1v_2)$ and
$C(v_1w)$ is saturated by $v_1$. Since $C(v_1v_3)$ is not saturated by 
$v_1$, we have $C(v_1w)\neq C(v_1v_3)$, and thus $w\neq v_3$. 
Now $C(wv_2)\neq C(v_1v_2)$ and $C(wv_2)\neq C(v_1w)$,
and $v_1v_2wv_1$ is a rainbow $C_3$. Hence there are two rainbow $C_3$'s,
a contradiction. Suppose that none of $\{C(v_1v_2),C(v_1v_3)\}$ is
saturated by $v_1$. There are $w,x\in N(v_1)$ such that $C(v_1w),C(v_1x)$
are saturated by $v_1$, so $C(v_1v_2)$, $C(v_1v_3)$, $C(v_1w)$
and $C(v_1x)$ are distinct. Moreover, $v_1wxv_1$ is a rainbow triangle.
Hence there are two rainbow
triangles in $G$, a contradiction. Thus, we have proved that $C(v_1v_2),C(v_1v_3)$
are saturated by the vertex $v_1$. Similarly, $C(v_2v_1),C(v_2v_3)$ are
saturated by $v_2$, and $C(v_3v_1),C(v_3v_2)$ are saturated by $v_3$.
Notice that $C(v_1v_2)$ is saturated by both $v_1$ and $v_2$.
Thus, $C(v_1v_2)$ appears only once in $G$. Similarly, we can see
that $C(v_1v_3)$ and $C(v_2v_3)$ appear only once in $G$.

By Lemma \ref{lemma_complete1}, since $G$ is complete, it is easy
to see that there is no edge $v_iw$ satisfying
$w\in V(G)\backslash\{v_1,v_2,v_3\}$ and $C(v_iw)$ is
saturated by $v_i$, for each $i=1,2,3$.
	
Let $G^*$ be the edge-colored graph obtained by replacing
the color of $v_1v_2$ by $C(v_1v_3)$. For any vertex
$w\in V(G)\backslash \{v_1,v_2,v_3\}$ and $i,j\in \{1,2,3\}$,
since $wv_iv_jw$ is not rainbow in $G$ and each color on $v_1v_2v_3v_1$
appears only once, $C(wv_i)=C(wv_j)$. Hence $wv_iv_jw$
is not rainbow in $G^*$. So, $G^*$ contains
no rainbow triangle and $c(G^*)=n-1$. It follows from
Theorem \ref{theorem_G0} that $G^*$ belongs to $\mathcal{G}_0$.
Thus there exists a partition $V=V_1\cup V_2$
(we can assume $\{v_1,v_2,v_3\}\subseteq V_1$), such
that $G^*[V_1,V_2]$ is monochromatic and $G^*[V_i]\in\mathcal{G}_0$
for $i=1,2$.
	
It is easy to see that $G[V_1,V_2]$ is monochromatic
and $G[V_2]=G^*[V_2]\in \mathcal{G}_0$. Moreover,
$c(G[V_1])=|G[V_1]|$ and $G[V_1]$ contains only one rainbow
triangle. By the induction hypothesis, $G[V_1]\in \mathcal{G}_1$.
It follows from the definition of $\mathcal{G}_1$ that
$G\in \mathcal{G}_1$.

The proof is complete. {\hfill$\Box$}

\section{Proofs of Theorems \ref{theorem_colorneighbor_C_4},
\ref{theorem_colorneighbor_C_3} and \ref{theorem_colorneighbor_verdiscyc}}
We need the following lemmas.
\begin{lemma}\label{Lemma_spabip}
Let $G$ be an edge-colored graph. Then $G$ contains a spanning
bipartite subgraph $H$ such that
$2d_{H}^{c}(v)+3d_{H}(v)\geq d_G^{c}(v)+d_G(v)$ for every
vertex $v\in V(H)$.
\end{lemma}

\begin{proof}
We choose a spanning bipartite subgraph $H$ of $G$ such that
$f(H):=e(H)+\sum_{v\in V(H)}d_{H}^{c}(v)$
is as large as possible. We will show that
$2d_{H}^{c}(v)+3d_{H}(v)\geq d_G^{c}(v)+d_G(v)$
for every vertex $v\in V(H)$.

Suppose that the bipartition of $H$ is $(X,Y)$. Then any edge $xy$
of $G$ with $x\in X$ and $y\in Y$ is also an edge of $H$. Otherwise,
$f(H+xy)> f(H)$, contradicting the choice of $H$.
One can see that $d^{c}_{H}(x)=|CN_{G[Y]}(x)|$ for $x\in X$,
and $d^{c}_{H}(y)=|CN_{G[X]}(y)|$ for $y\in Y$.

Suppose that there exists a vertex $u\in V(H)$ such that
\begin{align}
2d_{H}^{c}(u)+3d_{H}(u)<d_G^{c}(u)+d_G(u).
\end{align}
Without loss of generality, we may assume $u\in X$. We claim that $|X|\geq 2$.
Suppose that $X=\{u\}$. Since $e_G(X,Y)=e_H(X,Y)$, we get
$2d_{H}^{c}(u)+3d_{H}(u)\geq 2d_{H}^{c}(u)+3d_{G}(u)\geq d_G^{c}(u)+d_G(u)$,
a contradiction. This proves $|X|\geq 2$. Let $H'$ be the
spanning bipartite subgraph of $G$ with the bipartition
$(X\backslash \{u\}, Y\cup \{u\})$ and edge set
$E(H)\cup \{ux\in E(G): x\in X\setminus \{u\}\}\setminus \{uy\in E(G): y\in Y\}$.
Then
\begin{align}
e(H')-e(H)=(d_{G}(u)-d_{H}(u))-d_{H}(u)=d_{G}(u)-2d_{H}(u).
\end{align}
On the other hand, we obtain
\begin{align*}
d^{c}_{H'}(u)-d^{c}_{H}(u)
&=|CN_{G[X]}(u)|-|CN_{G[Y]}(u)|\\
&\geq |CN_{G}(u)|-2|CN_{G[Y]}(u)|\\
&=d^{c}_{G}(u)-2d^{c}_{H}(u),
\end{align*}
and
\begin{align*}
\sum_{v\in V(G)\setminus\{u\}}(d^{c}_{H'}(v)-d^{c}_{H}(v))
&=\sum_{v\in X\setminus\{u\}}(d^{c}_{H'}(v)-d^{c}_{H}(v))+\sum_{v\in Y}(d^{c}_{H'}(v)-d^{c}_{H}(v))\\
&\geq \sum_{v\in Y}(d^{c}_{H'}(v)-d^{c}_{H}(v))\\
&=\sum_{v\in Y}(|CN_{G[X\backslash \{u\}]}(v)|-|CN_{G[X]}(v)|)\\
&\geq -\sum_{v\in Y} |CN_{G[\{u\}]}(v)|=-d_{H}(u).
\end{align*}
Thus
\begin{align*}
\sum_{v\in V(G)}d^{c}_{H'}(v)-\sum_{v\in V(G)}d^{c}_{H}(v)
&=\sum_{v\in V(G)\setminus\{u\}}(d^{c}_{H'}(v)-d^{c}_{H}(v))+(d^{c}_{H'}(u)-d^{c}_{H}(u))\\
&\geq (d^{c}_{G}(u)-2d^{c}_{H}(u))-d_{H}(u),
\end{align*}
that is,
\begin{align}
\sum_{v\in V(G)}d^{c}_{H'}(v)-\sum_{v\in V(G)}d^{c}_{H}(v)\geq  d^{c}_{G}(u)-2d^{c}_{H}(u)-d_{H}(u).
\end{align}
By (1), (2) and (3), we get
$$
f(H')-f(H)\geq d_{G}(u)+d^{c}_{G}(u)-2d^{c}_{H}(u)-3d_{H}(u)>0,
$$
which contradicts the choice of $H$. The proof is complete.
\end{proof}

\begin{lemma}[\v{C}ada, Kaneko, Ryj\'{a}\v{c}ek, and Yoshimoto \cite{CKRY16}]\label{Lemma_CKRY}
Let $G$ be an edge-colored graph of order $n$. If $G$ is triangle-free and
$\delta^c(G)\geq \frac{n}{3}+1$, then $G$ contains a rainbow $C_4$.
\end{lemma}

\begin{lemma}\label{Lemma_DisRainC4}
Let $k\geq 1$ be an integer and $G$ an edge-colored graph of order $n\geq k+3$.
If $G$ is triangle-free and $\delta^c(G)\geq \frac{n}{3}+k$, then
$G$ contains $k$ rainbow $C_4$'s.
\end{lemma}
\begin{proof}
We prove this lemma by induction on $k$. The case $k=1$
is true by Lemma \ref{Lemma_CKRY}. Suppose that the
lemma holds for $k-1$. Let $v$ be a vertex
of a rainbow $C_4$ in $G$, and set $G':=G-v$.
Then $\delta^c(G')\geq \delta^c(G)-1\geq \frac{n}{3}+k-1>\frac{|G'|}{3}+(k-1)$.
By the induction hypothesis, there are $k-1$ rainbow $C_4$'s in $G'$,
and still in $G$. So, there are $k$ rainbow $C_4$'s in $G$.
\end{proof}

We point out that Lemma \ref{Lemma_DisRainC4} has the following extension.
This result can be proved by using Lemma \ref{Lemma_CKRY} and induction, we omit the proof
here.

\begin{prop}
Let $k\geq 1$ be an integer and $G$ an edge-colored graph of order $n\geq 4k$.
If $G$ is triangle-free and $\delta^c(G)\geq n/3+2(k-1)+1$, then
$G$ contains $k$ vertex-disjoint rainbow $C_4$'s.
\end{prop}

\begin{lemma}\label{Lemma_FiveVertices}
Let $G$ be an edge-colored graph of order $n$ such that
$\delta^c(G)=n-1$ (so $G$ is complete).
For any subset $S$ of $V(G)$ with $|S|=5$, $G[S]$ contains a rainbow $C_4$.
\end{lemma}
\begin{proof}We prove the lemma by contradiction. Suppose that $G[S]$ contains no rainbow $C_4$.
Let $S=\{x_1,x_2,x_3,x_4,x_5\}\subset V(G)$. Since $\delta^c(G)=n-1$,
any two incident edges have distinct colors in $G$.
Thus, we may assume that $G[S]$ contains two monochromatic
independent edges, say, $C(x_1x_2)=C(x_3x_4)=1$. Without loss of
generality, set $C(x_1x_5)=2$ and $C(x_3x_5)=3$. Since $G[S]$ contains no
rainbow $C_4$ and any two incident edges have distinct colors,
we obtain $C(x_2x_3)=2, C(x_1x_4)=3$, and moreover,
$C(x_2x_4)\notin \{1,2,3\}$, say, $C(x_2x_4)=4$.
Observing the colors on the edges incident to
$x_2$ and $x_5$, we see that $C(x_2x_5)\notin \{1,2,3,4\}$,
so set $C(x_2x_5)=5$. Consequently, there is a
rainbow $C_4$ with colors $1,3,4,5$ in $G[S\backslash \{x_1\}]$, a
contradiction.
\end{proof}

\smallskip
\noindent
{\bf {Proof of Theorem \ref{theorem_colorneighbor_C_4}.}}
When $\delta^c(G)=n-1$, it follows from Lemma~\ref{Lemma_FiveVertices} that
there are $k$ rainbow $C_4$'s in $G$, since the order $n\geq 105k-24\geq5k$.
Thus we may assume that $\delta^{c}(G)\leq n-2$.

Let $u$ be a vertex with $d_G^{c}(u)=\delta^{c}(G)$ and set
$t:=\delta^{c}(G)$. Let $T$ be a subset of $N_G(u)$ such that
$|T|=t$ and $C(ux)\neq C(uy)$ for every two vertices $x,y\in T$.
Without loss of generality, set $T=\{x_1,x_2,\ldots,x_t\}$
and assume that $C(ux_i)=i$ for $i\in \{1,2,\ldots,t\}$.
Set $G_1=G[T\cup \{u\}]$ and $G_2=G-G_1$.
Since $|G_1|=t+1\leq n-1$, $V(G_2)\neq \emptyset$.

First, suppose that there are $k$ vertices $z\in V(G_2)$ such
that $|CN_{G_1}(z)\setminus CN(u)|\geq 2$.
By the choice of $T$, if $v\in V(G_1)$ is a neighbor of $z$ such that
$C(vz)\in CN_{G_1}(z)\setminus CN(u)$, then $v\neq u$. Since
$|CN_{G_1}(z)\setminus CN(u)|\geq 2$, choose
$x_r,x_s\in T$ with $\{C(x_rz),C(x_sz)\} \subseteq CN_{G_1}(z)\backslash CN(u)$,
and $ux_rzx_su$ is a rainbow $C_4$. Thus, there are $k$ rainbow $C_4$'s.

Now, suppose that $|CN_{G_1}(v)\setminus CN(u)|\leq 1$ holds
for at least $n-t-k$ vertices $v\in V(G_2)$. We say that a vertex
$v\in V(G_2)$ is \emph{good} if $|CN_{G_1}(v)\setminus CN(u)|\leq 1$.
\setcounter{claim}{0}
\begin{claim}\label{CNG2}
$|CN_{G_2}(v)|=|G_2|-1$ for any good vertex $v\in V(G_2)$.
\end{claim}
\begin{proof}
First, $|CN_{G_1}(v)\setminus CN(u)|\leq 1$. It follows
from $|CN(u)|=t$ that $|CN(u)\cup CN_{G_1}(v)|\leq t+1$. Note
that $|CN(u)\cup CN(v)|\geq n-1$, we have
$|CN(v)\backslash CN_{G_1}(v)|\geq n-t-2$.
On the other hand,
$|CN(v)\backslash CN_{G_1}(v)|\leq |CN_{G_2}(v)|\leq d_{G_2}(v)\leq |G_2|-1=n-t-2$.
Thus, $|CN_{G_2}(v)|=|G_2|-1$, where $|G_2|=n-t-1$.
\end{proof}

Denote by $H'$ the subgraph induced by $n-t-k$ good vertices in $G_2$.
By Claim \ref{CNG2}, the underlying graph of $H'$ is complete. Furthermore,
for any vertex $v\in V(H')$, $d_{H'}^c(v)=|H'|-1$. First suppose that
$t\leq n-6k$. Note that $|H'|=n-t-k\geq 5k$.
Applying Lemma \ref{Lemma_FiveVertices} to $H'$,
we see that there are $k$ rainbow $C_4$'s in $G_2$,
which are also in $G$.

Thus we may assume $t\geq n-6k+1$.
By Lemma \ref{Lemma_spabip}, there is a spanning bipartite subgraph $H$
of $G$ such that
\begin{align}
2d_{H}^{c}(v)+3d_{H}(v)\geq d_G^{c}(v)+d_G(v)
\end{align}
for every vertex $v\in V(H)$. On the other hand,
since $H$ is a subgraph of $G$, it is not difficult to see that
\begin{align}
d_{H}(v)-d_{H}^{c}(v)\leq d_{G}(v)-d_{G}^{c}(v),
\end{align}
and
\begin{align}
d_{G}(v)-d_{G}^{c}(v)\leq d_{G}(v)-\delta^{c}(G)\leq (n-1)-(n-6k+1)=6k-2.
\end{align}
Together with (5) and (6),
\begin{align}
d_{H}^{c}(v)-d_{H}(v)\geq 2-6k.
\end{align}
Recall that $d_G^c(v)\geq \delta^{c}(G)=t\geq n-6k+1$, and
$d_G(v)\geq d_G^{c}(v)$. Then, combining (4) with (7), we obtain
$$
d^{c}_{H}(v)\geq \frac{1}{5}(d_G^{c}(v)+d_G(v)+6-18k)\geq \frac{2n-30k+8}{5}\geq\frac{n}{3}+k
$$
when $n\geq 105k-24$. By Lemma \ref{Lemma_DisRainC4}, there are
$k$ rainbow $C_4$'s in $H$, which are also $k$ rainbow $C_4$'s in
$G$. The proof of Theorem \ref{theorem_colorneighbor_C_4} is complete. {\hfill$\Box$}

\smallskip
\noindent
{\bf {Proof of Theorem \ref{theorem_colorneighbor_C_3}.}}
Suppose that $G$ contains no rainbow triangles.
First suppose that there exists a vertex, say $u$,
such that $d_G^c(u)\leq\frac{n-1}{2}$. For any vertex $v$ which
is adjacent to $u$, $|CN(u)\cup CN(v)|\geq n-1$. This
implies that
$$d_G^c(u)+d_G^c(v)=|CN(u)\cup CN(v)|+|CN(u)\cap CN(v)|\geq (n-1)+1=n.$$
It follows that $d_G^c(v)\geq \frac{n+1}{2}$ for any vertex $v$ adjacent to $u$.
For any vertex $v$ which is not adjacent to $u$, we also have
$|CN(u)\cup CN(v)|\geq n-1$. This implies
$d_G^c(u)+d_G^c(v)=|CN(u)\cup CN(v)|+|CN(u)\cap CN(v)|\geq n-1$. It
follows that $d_G^c(v)\geq \frac{n-1}{2}$ for any vertex $v$ not
adjacent to $u$.

Set $H:=G-u$. Then, we obtain $d^c_{H}(v)\geq d^c_G(v)-1\geq\frac{|H|}{2}$
for any vertex $v$ adjacent to $u$, and $d^c_{H}(v)\geq d^c_G(v)\geq\frac{|H|}{2}$
for any vertex $v$ not adjacent to $u$. By Theorem \ref{theorem_RainbowC_3},
the underlying graph of $H$ is isomorphic to $K_{\frac{n-1}{2},\frac{n-1}{2}}$,
where $n$ is odd. Let $(X,Y)$ be the bipartition of $H$, where $X=\{x_1,x_2,\ldots,x_t\}$,
$Y=\{y_1,y_2,\ldots,y_t\}$, $t=\frac{n-1}{2}$. We claim that $N_G(u)\subseteq X$ or
$N_G(u)\subseteq Y$. Suppose that $N_G(u)\cap X\neq \emptyset$ and
$N_G(u)\cap Y\neq \emptyset$. Without loss of generality, suppose
that $ux_1\in E(G)$ and $uy_1\in E(G)$. Since $d_{G}^c(x_1)\geq \frac{n+1}{2}=d_G(x_1)$
and $d_{G}^c(y_1)\geq \frac{n+1}{2}=d_G(y_1)$, we have
equality in both cases, and thus $C(x_1u)\neq C(x_1y_1)$
and $C(y_1u)\neq C(x_1y_1)$. This implies that $C(x_1u)=C(y_1u)$.
We also can derive that all edges
incident to $u$ have the same color, that is, $d^c_{G}(u)=1$.
For two vertices $x,u$, $|CN(u)\cup CN(x_1)|=|CN(x_1)|=\frac{n+1}{2}<n-1$
when $n\geq 4$, a contradiction. Thus, we have shown that
$N_G(u)\subseteq X$ or $N_G(u)\subseteq Y$. Without loss of
generality, suppose that $N_G(u)\subseteq X$. For any vertex
$v\in Y$, we have $|CN_G(u)\cup CN_G(v)|=n-1$
and $|CN_G(v)|=|X|=\frac{n-1}{2}$. Thus, $|CN_G(u)|=\frac{n-1}{2}$
and $CN_G(u)\cap CN_G(v)=\emptyset$. This implies that the underlying
graph of $G$ is $K_{\frac{n+1}{2},\frac{n-1}{2}}$. For any two vertices
$v_1,v_2\in Y$, by the condition $|CN(v_1)\cup CN(v_2)|\geq n-1$, we can
derive that any two edges incident to $v_1$ or $v_2$ have distinct colors.
Since $v_1,v_2\in Y$ are chosen arbitrarily, $G$ is a rainbow
$K_{\frac{n+1}{2},\frac{n-1}{2}}$.

Now assume that $d^c_G(v)\geq \frac{n}{2}$ for any vertex $v\in V(G)$.
By Theorem \ref{theorem_RainbowC_3}, $n$ is even and the underlying
graph of $G$ is $K_{\frac{n}{2},\frac{n}{2}}$. Arguing similarly
as above, we see that $G$ is a rainbow $K_{\frac{n}{2},\frac{n}{2}}$.
The proof is complete. {\hfill$\Box$}

Let $D$ be a digraph with the vertex set $V(D)$ and arc set
$A(D)$. For $v\in V(D)$, the \emph{out-degree} of $v$ in $D$,
denoted by $d^+_D(v)$, is the number of out arcs from $v$.

\begin{lemma}[Alon \cite{A96}]\label{lem_a96}
Every digraph with minimum out-degree at least $64k$ contains
$k$ vertex-disjoint directed cycles.
\end{lemma}
\smallskip
\noindent
{\bf {Proof of Theorem \ref{theorem_colorneighbor_verdiscyc}.}}
By contradiction, suppose that $G$ contains no $k$ vertex-disjoint rainbow cycles.
Let $G_1,G_2,\cdots,G_r$ be $r$ vertex-disjoint rainbow cycles in $G$,
where $|G_i|\in \{3,4,5\}$ (possibly, $r=0$). We may assume that $G_1,G_2,\ldots, G_r$
are chosen so that $r$ is as large as possible.  Obviously, $r\leq k-1$.
Let $H:=G_1\cup G_2\cup\ldots G_r$, and $G':=G-V(H)$. Note that $0\le |H|\le 5r$.

Now choose $u,v\in V(G')$ with $uv\in E(G)$, and
$S_1=\{x_1,x_2,\ldots,x_{s_1}\}\subset N_{G'}(u)\backslash\{v\}$ and
$S_2=\{y_1,y_2,\ldots,y_{s_2}\}\subset N_{G'}(v)\backslash\{u\}$,
so that the following two conditions hold:
\begin{description}
\item{(1)} for any $1\leq i<j\leq s_1$, $C(x_iu)\neq C(x_ju), C(x_iu)\neq C(uv)$;
for any $1\leq i<j\leq s_2$, $C(y_iv)\neq C(y_jv), C(y_iv)\neq C(uv)$;
and for any $i\in \{1,2,\ldots,s_1\}$, $j\in \{1,2\ldots,s_2\}$,
$C(x_iu)\neq C(y_jv)$; and,
\item{(2)} subject to (1), $s_1+s_2$ is maximized.
\end{description}
Since $G'$ contains no rainbow $C_3$, $S_1\cap S_2=\emptyset$.
Set $G^{*}:=G[S_1\cup S_2\cup \{u,v\}]$. Note that
$$s_1+s_2+1=|CN_{G'}(u)\cup CN_{G'}(v)|\geq |CN(u)\cup CN(v)|-2|H|\geq n/2+64k+1-2|H|,$$
and $$|G^{*}|=s_1+s_2+2\geq n/2+64k+2-2|H|.$$

In what follows, we construct a digraph $D$ from $G^{*}$
by the following operations:
\begin{description}
\item{$(a)$} Set $V(D)=S_1\cup S_2$;
\item{$(b)$} For any pair of vertices $x_i,x_j\in S_1$ with $x_ix_j\in E(G)$,
$x_ix_j\in A(D)$ if $C(x_ix_j)=C(ux_j)$; and $x_jx_i\in A(D)$
if $C(x_ix_j)=C(ux_i)$;
\item{$(c)$} For any pair of vertices $y_i,y_j\in S_2$ with $y_iy_j\in E(G)$,
$y_iy_j\in A(D)$ if $C(y_iy_j)=C(vy_j)$; and $y_jy_i\in A(D)$ if $C(y_iy_j)=C(vy_i)$;
\item{$(d)$} For any pair of vertices $x_i\in S_1,y_j\in S_2$ with $x_iy_j\in E(G)$,
$C(x_iy_j)\in \{C(ux_i),C(vy_j),\break C(uv)\}$, or
there is a rainbow $C_4$. If $C(x_iy_j)=C(uv)$, then we do
not add an arc to $D$;
if $C(x_iy_j)=C(ux_i)$ then $y_jx_i\in A(D)$; and
if $C(x_iy_j)=C(vy_j)$ then $x_iy_j\in A(D)$.
\end{description}
By the construction, note that there is a directed cycle in $D$
if and only if there is a rainbow cycle in $G^{*}$. Furthermore,
if there are $(k-r)$ vertex-disjoint directed cycles in $D$,
then there are $(k-r)$ vertex-disjoint rainbow cycles in $G^{*}$, and
together with the $r$ vertex-disjoint rainbow cycles, this contradicts
the assumption that $G$ does not contain $k$ vertex-disjoint
rainbow cycles. Thus, there are no $(k-r)$ vertex-disjoint
directed cycles in $D$. By Lemma \ref{lem_a96}, we can
see there is a vertex, say $w_1\in S_1\cup S_2$, such that
$d^+_D(w_1)\leq 64(k-r)-1$. If $d^+_D(u)\geq 64(k-r)+1$
for any vertex $u\in V(D)\backslash \{w_1\}$,
then $d^+_{D'}(u)\geq 64(k-r)$, in which $D':=D-w_1$. By Lemma \ref{lem_a96}, there
are $k-r$ directed cycles in $D$, and $k$ rainbow cycles in $G$,
a contradiction. Thus, there are two vertices, say
$w_1,w_2\in S_1\cup S_2$, such that $d^+_D(w_1)\leq 64(k-r)-1$
and $d^+_D(w_2)\leq 64(k-r)$.
\setcounter{claim}{0}
\begin{claim}\label{newclaim}
$|G-(V(G^{*})\cup V(H))|\ge n/2+64k-2|H|-128(k-r)-1$.
\end{claim}
\begin{proof}
We divide the proof into two cases.

First, we assume that $w_1,w_2$ belong to a same set of $S_1,S_2$,
say, $w_1,w_2\in S_1$. In this case, we know that
all edges incident to $w_1$ or $w_2$ in $G^{*}$ can have at most
$3+(128(k-r)-1)$ colors, where the term $3$ comes from the fact
that $uw_1,uw_2$, together with the possibly existing edge
incident to $w_1$ or $w_2$ with the color $C(uv)$, correspond
to three colors. Since $|CN(w_1)\cup CN(w_2)|\geq \frac{n}{2}+64k+1$,
there are at least
$$n_1:=n/2+64k+1-2|H|-3-(128(k-r)-1)=n/2+64k-2|H|-128(k-r)-1$$
colors between $\{w_1,w_2\}$ and $V(G-G^{*}-H)$ in $G$. Let $C^*$
be the set of these $n_1$ colors. Notice that
$C^*\subset CN_{G'-G^{*}}(w_1)\cup CN_{G'-G^{*}}(w_2)$. For any
vertex $w'\in V(G')\backslash V(G^{*})$ such that $w_1w',w_2w'\in E(G)$
and $\{C(w_1w'), C(w_2w')\}\cap \{C(uw_1), C(uw_2)\}=\emptyset$,
it follows from $G'$ contains no rainbow $C_4$ that $C(w_1w')=C(w_2w')$.
Furthermore, every common neighbor of $w_1,w_2$ in $G'-G^{*}$ with
the color in $C^*$ must correspond to one new color. Thus,
there are at least $n/2+64k-2|H|-128(k-r)-1$ vertices in $G-(V(G^{*})\cup V(H))$.

Thus, we may assume that $w_1,w_2$ belong to different sets, say, $w_1\in S_1$
and $w_2\in S_2$. In this case, we know that all edges incident to
$w_1$ or $w_2$ in $G^{*}$ can have at most $3+(128(k-r)-1)$
colors, where the term $3$ comes from the fact that $uw_1,vw_2$,
together with the possible existing edge incident to $w_1$ or $w_2$ with
the color $C(uv)$, correspond to three colors. So, there are at least
$$n/2+64k+1-2|H|-3-(128(k-r)-1)=n/2+64k-2|H|-128(k-r)-1$$ colors in
$C^*=CN_{G'-G^{*}}(w_1)\cup CN_{G'-G^{*}}(w_2)$. For any vertex
$w'\in V(G')\backslash V(G^{*})$ such that $w_1w',w_2w'\in E(G)$ and
$\{C(w_1w'),C(w_2w')\}\cap \{C(uw_1),C(vw_2),C(uv)\}=\emptyset$,
it follows from $G'$ contains no rainbow $C_5$ that $C(w_1w')=C(w_2w')$.
Thus, every common neighbor of $w_1,w_2$
in $G'-G^{*}$ with the color in $C^*\backslash\{C(uw_1),C(vw_2),C(uv)\}$
corresponds to one new color. Thus, there are at least
$n/2+64k-2|H|-128(k-r)-1$ vertices in $G-(V(G^{*})\cup V(H))$.
\end{proof}

By Claim~\ref{newclaim},
\begin{align*}
|G|&=|G^{*}|+|H|+|G-(V(G^{*})\cup V(H))|\\
&\geq n/2+64k+2-2|H|+|H|+n/2+64k-2|H|-128(k-r)-1\\
&=n+128k-3|H|-128(k-r)+1\\
&\geq n+113r+1\\
&\geq n+1,
\end{align*}
a contradiction. The proof of Theorem
\ref{theorem_colorneighbor_verdiscyc} is complete. {\hfill$\Box$}

\begin{remark}
Bermond and Thomassen \cite{BT91} conjectured that every directed
graph with minimum out-degree at least $2k-1$ contains $k$
vertex-disjoint directed cycles. Alon \cite{A96} gave a linear
bound by proving that $64k$ suffices (Lemma~\ref{lem_a96}).
Recently, Buci\'c \cite{B18} proved a better bound $18k$
towards this conjecture. One may find that if we apply Buci\'c's
new bound instead of Alon's bound to our proof of Theorem
\ref{theorem_colorneighbor_verdiscyc}, then we can improve
the constant in the second term of
Theorem \ref{theorem_colorneighbor_verdiscyc}.
\end{remark}

\section{Concluding remarks}
Extending Mantel's theorem, Erd\H{o}s \cite{E55} proved
that a graph of order $n$ and size $\geq \lfloor\frac{n^2}{4}\rfloor+l$
contains at least $l\lfloor n/2\rfloor$ triangles, provided
$l\leq 3<n/2$. Erd\H{o}s \cite{E62} further conjectured
that the same conclusion holds when $l<n/2$. A slightly weaker
form of Erd\H{o}s' conjecture was proved by Lov\'asz and
Simonovits \cite{LS76}. (See also Bollob\'as \cite[pp.302]{B78}.)
One may ask for the rainbow version of Erd\H{o}s' conjecture. Furthermore,
we can pose the following related problem.
\begin{prob}
Let $k\geq 1$ be an integer. Let $G$ be an edge-colored graph of order $n$. Determine an
integer valued function $f(k)$ as small as possible, such that
if $e(G)+c(G)\geq n(n+1)/2+f(k)$ and $n$ is sufficiently large,
then $G$ contains at least $k$ rainbow $C_3$'s.
\end{prob}

Recently, Xu et al. \cite{XHWZ16} proved a rainbow version of Tur\'an's theorem.
Maybe it is also interesting to characterize the extremal graphs in their main
theorem.

Furthermore, our Lemma \ref{Lemma_spabip} is motivated by the following theorem
due to Erd\H{o}s.
\begin{thm}[Erd\H{o}s \cite{E65}]
Let $G$ be a graph. Then $G$ contains a spanning bipartite subgraph $H$,
such that $d_H(v)\geq \frac{1}{2}d_G(v)$ for all vertices $v\in V(G)$.
\end{thm}

We can naturally consider the counterpart of of Erd\H{o}s' theorem for edge-colored
graphs. Indeed, our Lemma \ref{Lemma_spabip} can be regarded as our attempt in 
this viewpoint. Along this line, it might be interesting to consider a degree
condition for the existence of rainbow (or properly colored) spanning
bipartite subgraphs in edge-colored graphs.

\section*{Acknowledgements}
The first author is supported by JSPS KAKENHI (No.\ 15K04979).
The second author is supported by NSFC (No.\ 11601379).
The third author is supported by NSFC (No.\ 11701441) and
the Fundamental Research Funds for the Central Universities
(No. XJS17027). The fourth author is supported by NSFC (No.\ 11671320).
The authors are very indebted to an anonymous referee for
his/her suggestions which largely improve the presentation of
this paper.

\end{document}